\documentclass[12pt]{amsart}

\usepackage{latexsym,amsmath,amssymb,amsthm,amsfonts}

\usepackage[numeric]{amsrefs}

\usepackage{hyperref}

\usepackage[capitalise]{cleveref}

\usepackage{multido,pst-plot,pstricks,pst-node}
\newcommand{\R}{{\mathbb{R}}}

\renewcommand{\S}{{\mathbb{S}}}
\newcommand{\eps}{\varepsilon}

\newcommand{\E}{{\mathbb{E}}}
\newcommand{\F}{{\mathcal{F}}}

\newcommand{\B}{{\mathbb{B}}}

\newcommand{\Sp}{{\rm Sp\,}}

\renewcommand{\phi}{{{\varphi}}}

\newtheorem{theorem}{Theorem}
\newtheorem{lemma}[theorem]{Lemma}
\newtheorem{proposition}[theorem]{Proposition}

\newtheorem{corollary}[theorem]{Corollary}

\theoremstyle{remark}
\newtheorem{remark}[theorem]{Remark}

\numberwithin{equation}{section}

\title{On a Gallai-type problem and illumination of spiky balls and cap bodies}

\author{Andrii Arman}
\address{Department of Mathematics, University of Manitoba, Winnipeg, MB, R3T 2N2, Canada}
\email{andrew0arman@gmail.com}
\thanks{The first author was supported in part by the Pacific Institute for the
	Mathematical Sciences Postdoctoral Fellowship.}

\author{Andriy Bondarenko}
\address{Department of Mathematical Sciences, Norwegian University of Science and Technology, NO-7491 Trondheim, Norway}
\email{andriybond@gmail.com}
\thanks{The second author was supported by grant 334466 of the Research Council of Norway ``Fourier Methods and Multiplicative Analysis''.}

\author{Andriy Prymak}
\address{Department of Mathematics, University of Manitoba, Winnipeg, MB, R3T 2N2, Canada}
\email{prymak@gmail.com}
\thanks{The third author was supported by NSERC of Canada Discovery Grant RGPIN-2020-05357.}

\author{Danylo Radchenko}
\address{Laboratoire Paul Painleve, University of Lille, F-59655 Villeneuve d'Ascq, France}
\email{danradchenko@gmail.com}
\thanks{The fourth author was supported by ERC Starting Grant No. 101078782.}



\begin{document}

\begin{abstract}
	We show that any finite family of pairwise intersecting balls in $\E^n$ can be pierced by $(\sqrt{3/2}+o(1))^n$ points improving the previously known estimate of $(2+o(1))^n$. As a corollary, this implies that any $2$-illuminable spiky ball in $\E^n$ can be illuminated by $(\sqrt{3/2}+o(1))^n$ directions. For the illumination number of convex spiky balls, i.e., cap bodies, we show an upper bound in terms of the sizes of certain related spherical codes and coverings. For large dimensions, this results in an upper bound of $1.19851^n$, which can be compared with the previous $(\sqrt{2}+o(1))^n$ established only for the centrally symmetric cap bodies. We also prove the lower bounds of $(\tfrac{2}{\sqrt{3}}-o(1))^n$ for the three problems above.
\end{abstract}

\maketitle

\section{Introduction and main results}
 Let $\mathcal{F}$ be a collection of sets, we say that a set of points $P$ pierces $\mathcal{F}$ (or that $P$ is a piercing set for $\F$) if for any $F\in \mathcal{F}$ we have $P\cap F\neq \emptyset$. Finding the smallest piercing set for a given collection $\mathcal{F}$ is a hard problem, even when $\mathcal{F}$ satisfies some additional intersection properties. We refer an interested reader to the survey~\cite{Ec} about the so called Hadwiger-Debrunner $(p,q)$-type questions, and for most of this paper we concentrate on the case when $\mathcal{F}$ is a collection of pairwise intersecting balls in $\E^{n}$. 
 
 Gallai raised a question of finding the minimal cardinality of a piercing set for arbitrary finite family of pairwise intersecting planar discs. This planar problem was solved by Danzer~\cite{Da86}, who showed that for any such collection of discs, there is a piercing set of cardinality at most $4$, and that $4$ is the best possible. See also~\cite{Grun} for some related results in the plane.
 
 More generally, define $G_n$ as the smallest integer such that for any finite family $\F$ of pairwise intersecting closed balls in $\E^n$ there exists a piercing set $P$ of cardinality $G_n$. The best known asymptotic upper bound on $G_n$ is $G_n\le (2+o(1))^n$, see~\cite{Ec}*{(4.4)}. This bound was obtained through Danzer's reduction to the spherical coverings by caps of angular radius $\pi/6$, and subsequent use of Rogers's estimate on the number of such caps needed to cover the entire sphere. 
 
In this paper we obtain new bounds on $G_n$. The main idea is to reduce the problem to the situation when the balls have almost equal radii and subsequently invoke the result of Bourgain and Lindenstrauss~\cite{Bo-Li} on covering a set by balls of the same diameter. Consequently, we obtain a much better exponent.
	\begin{theorem}
		\label{thm:gal-upper}
		$G_n\le (\sqrt{3/2}+o(1))^n$, i.e., any finite family of pairwise intersecting closed balls in $\E^n$ can be pierced by $(\sqrt{3/2}+o(1))^n$ points.
	\end{theorem}

We now proceed with preliminaries needed to discuss applications of Theorem~\ref{thm:gal-upper} to  certain illumination problems.
	
For a compact $K$ in $\E^n$ a point $x$ on the boundary of $K$ is illuminated by a direction $u\in\S^{n-1}$, where $\S^{n-1}$ is the unit sphere in $\E^n$, if the half line originating at $x$ in the direction of $u$ intersects the interior of $K$ at points arbitrarily close to $x$. The illumination number $I(K)$ of $K$ is the smallest number of directions needed to illuminate all boundary points of $K$. A convex body in $\E^n$ is a convex compact set in $\E^n$ with non-empty interior. A major open question in discrete and convex geometry is to determine the largest value of $I(K)$ where $K$ is a convex body in $\E^n$. 
 
 It was conjectured by Hadwiger~\cite{Ha} and Boltianski~\cite{Bo} that this value is $2^n$ with equality attained only for affine copies of the hypercube. An interested reader is referred to the survey~\cite{Be-Kh}, however, several new results have been obtained since then, including the currently best known asymptotic upper bound $\exp(\tfrac{-cn}{\log^8n})4^n$ in~\cite{CHMT} and some results in small dimensions~\cite{Pr}, \cite{Pr-Sh}, \cite{ABP-small-dim}. Verification of the conjecture, or estimates of the illumination number, for particular classes of convex bodies or compact sets are also of great interest. Here we will study illumination of cap bodies and spiky balls continuing research started in a recent work~\cite{Be-cap} of Bezdek, Ivanov and Strachan.
		
Let $\B^n$ denote the unit ball in $\E^n$, $\|\cdot\|$ the Euclidean norm, ${\rm conv}(\cdot)$ the convex hull, $C(x,\alpha)=\{y\in\S^{n-1}:x\cdot y>\cos\alpha\}$ the open spherical cap with center $x\in\S^{n-1}$ of angular radius $\alpha\in(0,\pi)$, and $ C[x,\alpha]=\{y\in\S^{n-1}:x\cdot y\ge\cos\alpha\}$ be the corresponding closed cap. 
 
For any points $x_1,\dots,x_m\in\E^n\setminus\B^n$, we define a spiky ball as
	\[
	\Sp[x_1,\dots,x_m]:=\bigcup_{i=1}^m{\rm conv}\,(\B^n\cup\{x_i\}).
	\]
Without loss of generality, we will assume below that each $x_i$ is a vertex of the spiky ball, i.e., $\Sp[x_1,\dots,x_m]\ne \Sp[x_1,\dots,x_{i-1},x_{i+1},\dots,x_m]$. Each ``spike'' has the base $B(x_i):=\S^{n-1}\cap \Sp[x_i]= C[\tfrac{x_i}{\|x_i\|},\arccos\tfrac{1}{\|x_i\|}]$. We will refer to this cap as the cap associated with $x_i$. 
	
Coverings of the sphere arise naturally in estimates of the illumination numbers of spiky balls. For $0<\theta<\tfrac\pi2$ and $A\subset\S^{n-1}$, the covering number $N(A,n,\theta)$ is the smallest number of closed spherical caps of angular radius $\theta$ required to cover $A$. We set $N(n,\theta):=N(\S^{n-1},n,\theta)$, and it is known (e.g., Rogers~\cite{Ro} or B\"or\"oczky-Wintsche~\cite{Bo-Wi}) that for a fixed $\theta$, as $n\to \infty$ we have 
\begin{equation}\label{eq:BW}N(n, \theta)=\left(\frac{1}{\sin \theta} +o(1)\right)^{n}.\end{equation}
	
In general, a spiky ball may have arbitrarily large illumination number. A natural class of spiky balls was suggested in~\cite{Be-cap} by Bezdek, Ivanov and Strachan. Namely, a spiky ball is 2-illuminable if any two of its vertices can be illuminated by a single direction. With
	\[
	I_{s,n}:=\max\{I(K):K\text{ is $2$-illuminable spiky ball in }\E^n\},
	\]
it was shown in~\cite{Be-cap} that $I_{s,n}\le 3+N(n-1,\pi/6)$ for $n\ge 4$, which implies $I_{s,n}<(2+o(1))^n$. In fact, using a stereographic projection, it was shown in the proof of~\cite{Be-cap}*{Th.~2(iii)} that $I_{s,n}\le 2+G_{n-1}$, and then Danzer's bound $G_{n-1}\le 1+N(n-1,\pi/6)$ was applied. Combining the inequality $I_{s,n}\le 2+G_{n-1}$ with \cref{thm:gal-upper} immediately implies the following improvement.
	\begin{corollary}
		$I_{s,n}\le (\sqrt{3/2}+o(1))^n$.
	\end{corollary} 
	
Let us now discuss lower bounds on $I_{s,n}$ and $G_n$. Let $B_n$ be the smallest integer such that any finite point set of diameter $1$ in $\E^n$ can be covered by $B_n$ closed balls of diameter $1$. It is rather immediate to observe that $G_n\ge B_n$, while the inequality $I_{s,n}\ge 1+B_{n-1}$ was established in the proof of~\cite{Be-cap}*{Rem.~5}. We have recently showed~\cite{ABP}*{Th.~2} that $B_n\ge(2/\sqrt{3}-o(1))^n$. Therefore, we have the following lower bounds.
	\begin{corollary}
	    $G_n\ge (\frac{2}{\sqrt{3}}-o(1))^n$ and $I_{s,n}\ge (\frac{2}{\sqrt{3}}-o(1))^n$.
	\end{corollary}
Numerically, $\frac{2}{\sqrt{3}}\approx1.1547$ while $\sqrt{\tfrac{3}{2}}\approx 1.2247$. It may be a challenging problem to find the asymptotic behaviour of both $G_n$ and $I_{s,n}$, or even to obtain an exponential improvement of the presented above lower/upper bounds.
	
%
	
We proceed with illumination of cap bodies. Cap body is a spiky ball which is a convex body. Similarly to $I_{s,n}$, we introduce 
	\begin{align*}
	I_{c,n}&:=\max\{I(K):K\text{ is a cap body in }\E^n\}, \quad\text{and}	\\
	I_{c,n}^{(s)}&:=\max\{I(K):K\text{ is a centrally symmetric cap body in }\E^n\}.
	\end{align*}
Clearly, $I_{c,n}^{(s)}\le I_{c,n}$. The estimate $I_{c,n}^{(s)}\le 2+N(n-1,\pi/4)$, $n\ge 4$, implying the asymptotic $I_{c,n}^{(s)}\le (\sqrt{2}+o(1))^n$, was obtained  in~\cite{Be-cap}.  
	
 In order to obtain new upper bounds on $I_{c,n}$, we use a combination of covering and packing (spherical codes) bounds. For $0<\theta<\pi$, let $M(n,\theta)$ denote the maximal number of non-overlapping open spherical caps of $\S^{n-1}$ with angular radius $\theta/2$ (having $\theta$-separated centers).
	\begin{theorem}\label{thm:cap-bodies-ill}
		For all $\alpha\in (0, \pi/2)$ we have $\displaystyle
		I_{c,n}\le M(n,2\alpha)+N(n,\pi/2-\alpha).$
	\end{theorem}
      Using the best known asymptotic upper bound on spherical codes due to Kabatjanskii and Levenstein~\cite{KL}, as well as the classical Roger's estimate~\cite{Ro} on spherical covers, we obtain a new upper bound on $I_{c,n}$, and, consequently, an exponential improvement on the upper bound for $I_{c,n}^{(s)}$.
	\begin{theorem}\label{cor:cap-bodies}
	The estimate
	\[
	I_{c,n}< (\tfrac1{\cos\alpha}+o(1))^n
	\]
	holds,
	where $\alpha\approx 0.583808$ is the solution of the equation
	\begin{equation}\label{eqn:alpha}
	\frac1{\cos\alpha}=\left(\frac{1+\sin2\alpha}{2\sin2\alpha}\right)^{\tfrac{1+\sin2\alpha}{2\sin2\alpha}}\left(\frac{2\sin2\alpha}{1-\sin2\alpha}\right)^{\tfrac{1-\sin2\alpha}{2\sin2\alpha}}, \quad 0<\alpha<\frac\pi4.
	\end{equation}
	In particular, $I_{c,n}<1.19851^n$, $n\ge n_0$.
	\end{theorem}

	Both $I_{c,n}$ and $I_{c,n}^{(s)}$ grow exponentially in $n$. Nasz\'{o}di~\cite{Na} showed $I_{c,n}^{(s)}\ge (1.116-o(1))^n$. Using the construction from~\cite{ABP}*{Th.~2}, we improve this bound as follows. 
	\begin{theorem}\label{thm:lower-cap}
		$I_{c,n}^{(s)}\ge (\frac{2}{\sqrt{3}}-o(1))^n$.
	\end{theorem}
	As for $G_n$ and $I_{s,n}$, it is compelling to find the asymptotic behaviour of $I_{c,n}$ and $I_{c,n}^{(s)}$, or to get a better exponential bounds and narrow the gap between $\frac{2}{\sqrt{3}}\approx1.1547$ and $1.19852$.
	
	We prove \cref{thm:gal-upper} in \cref{sec:Gallai},  \cref{thm:cap-bodies-ill} and \cref{cor:cap-bodies} in \cref{sec:Illum}, and \cref{thm:lower-cap} in \cref{sec:IllumLB}.

\section{Upper bound on Gallai's number $G_n$}\label{sec:Gallai}


Throughout this section we assume $n\ge 2$.

\begin{lemma}
	\label{lem:sphere-cap-overlap}
	If $\B^n$ and $r\B^n+x$ intersect, where $r\ge2$, then $(2\S^{n-1})\cap (r\B^n+x)$ contains a closed spherical cap $2C[y,\alpha]$ of angular radius  $\alpha\geq \arccos\tfrac{2r+5}{4(r+1)}$.
\end{lemma}

\begin{proof}
	If $x=0$, then the conclusion trivially holds. Otherwise, we can move the ball $r\B^n+x$ away from the origin until it touches $\B^n$, namely, we replace $x$ with $(r+1)\tfrac{x}{\|x\|}$. With this change, the intersection $(2\S^{n-1})\cap (r\B^n+x)$ only gets smaller. Then we have $(2\S^{n-1})\cap (r\B^n+x)=2C[\tfrac{x}{\|x\|},\alpha]$, where $\alpha$ satisfies
	$
	(2)^2+(r+1)^2-4(r+1)\cos\alpha=r^2,
	$
	leading to the required bound.
\end{proof}

\begin{lemma}
	\label{lem:small-ball-cover}
	If $\sqrt{1-\tfrac1n}\le \tfrac{r_1}{r_2}\le 1$, $r_1,r_2>0$, then any ball of radius $r_2$ in $\E^n$ can be covered by $2n$ balls of radius $r_1$. 
\end{lemma}
\begin{proof}
    Applying the appropriate homothety, we can assume that $r_2=1$. Place the centers of the smaller balls at the $2n$ points, whose one coordinate is $\pm\frac1{\sqrt{n}}$ and the other coordinates are zeroes. Let $x=(x_1,\dots,x_n)\in\B^n$. By symmetry, we can assume that $|x_1|=\max_{i \in \{1, \ldots, n\}} |x_i|$ and $x_1>0$, which gives
    \[
    \|x-(\tfrac1{\sqrt{n}},0,\dots,0)\|^2 = \|x\|^2-2x_1\frac1{\sqrt{n}}+\frac1n
    \le \|x\|^2-\frac2{n}\|x\|+\frac1n\le 1-\frac1n,
    \]
    implying the required covering property.
\end{proof}

\begin{proof}[Proof of \cref{thm:gal-upper}.]
	Let $\F$ be the given family of balls. We can assume that the smallest ball in $\F$ is $\B^n$. For an interval $I\subset [1,\infty)$, let $\F_I$ denote the sub-collection of $\F$ consisting of all balls with radius in $I$. 
 
    We first pierce all balls from $\F_{[n,\infty)}$. Since $\arccos\tfrac{2n+5}{4(n+1)}=\tfrac\pi3-O(\tfrac1n)$, by~(\ref{eq:BW}), there exists a cover of the sphere $2\S^{n-1}$ by $a_n \le (\tfrac{2}{\sqrt{3}}+o(1))^n$ closed caps of angular radius $\arccos\tfrac{2n+5}{4(n+1)}$. Let $C_0$ be centers of these caps. By~\cref{lem:sphere-cap-overlap}, any $r\B^{n}+x$ with $r\geq n$ contains a cap of angular radius at least $\arccos\tfrac{2n+5}{4(n+1)}$, and such cap contains at least one point of $C_0$. So $C_0$ pierces the set $\F_{[n,\infty)}$.
	
    Now, denote $\lambda:=(1-\tfrac{1}{n})^{-1/2}$, then $\lambda>1$ and let $t$ be the smallest integer with $\lambda^t>n$. For each $k$, $1\leq k \leq t$, consider the set $X_k$ of the centers of the balls from $\F_{[\lambda^{k-1},\lambda^k)}$. Since any two balls in $\F_{[\lambda^{k-1},\lambda^k)}$ intersect, the diameter of $X_k$ is at most $2\lambda^k$. By~\cite{Bo-Li}, any set of diameter $d$ in $\R^{n}$ can be covered by at most $(\sqrt{3/2}+o(1))^n$ balls of diameter at most $d$, so it is possible to cover $X_k$ by at most $b_n\le (\sqrt{3/2}+o(1))^n$ balls of radius $\lambda^k$. Now, by \cref{lem:small-ball-cover}, every ball of radius $\lambda^{k}$ can be covered by $2n$ balls of radius $\lambda^{k-1}$. So $X_k$ can be covered by $2nb_n$ balls of radius $\lambda^{k-1}$ with centers of these balls forming a set $C_k$. So every ball in $\F_{[\lambda^{k-1},\lambda^k)}$ will contain at least one point of $C_k$, so $C_k$ pierces $\F_{[\lambda^{k-1},\lambda^k)}$.
	
	All together, for $\F=\F_{[1,\infty)}$ we constructed a piercing set $C_0\cup C_1\cup \ldots C_{t}$ of cardinality $a_n+2ntb_n$, which, by $t\le (2+o(1))n \log n$, gives the desired bound. 
\end{proof}

\section{Upper bounds on illumination of cap bodies}\label{sec:Illum}

A spiky ball $\Sp[x_1,\dots,x_m]$ is a cap body if and only if the open caps associated with $x_i$ are mutually disjoint. The following proposition, which we will only use for the cap bodies, is straightforward and can also be found in~\cite{Be-cap}*{Lemma 10 (b)}. Here the positive hull of vectors $\{y_i\}_{i=1}^{k}$ is $\{\sum_{i=1}^{k} \lambda_i y_{i} \;:\; \lambda_i\in(0,\infty)\}$.
\begin{proposition}\label{prop:illum}
	A spiky ball\ \ $\Sp[x_1,\dots,x_m]$ is illuminated by directions $\{y_j\}_{j=1}^k\subset\S^{n-1}$ provided the positive hull of $\{y_j\}_{j=1}^k$ is $\E^{n}$ and $C(-\tfrac{x_i}{\|x_i\|},\tfrac\pi2-\arccos\tfrac{1}{\|x_i\|})\cap \{y_j\}_{j=1}^k\ne\emptyset$ for each $i$, $1\le i\le m$.
\end{proposition}

Our approach to illumination of cap bodies is the following. The vertices which are ``close'' to the sphere $\S^{n-1}$ can be illuminated by a system of directions which are centers of the corresponding covering system of caps (independent of the vertices). The vertices which are ``far'' from the sphere can be illuminated using one direction per each vertex. As the latter vertices must be separated due to convexity, their number can be estimated using bounds on $M(n,\theta)$. 
 
\begin{proof}[Proof of \cref{thm:cap-bodies-ill}.]
	Let $K=\Sp[x_1,\dots,x_m]$ be a cap body. We can assume
	\[
	\|x_1\|\ge \|x_2\|\ge \dots\ge \|x_k\|\ge \frac1{\cos\alpha}>\|x_{k+1}\|\ge \dots\ge \|x_m\|.
	\]	
	As $K$ is convex, $C(\tfrac{x_i}{\|x_i\|},\arccos\tfrac1{\|x_i\|})$, $1\le i\le k$, are mutually disjoint. Hence, any two points from $U_1:=\{-\tfrac{x_i}{\|x_i\|}\}_{i=1}^k$ are at least $2\alpha$ apart. Thus, $|U_1|\le M(n,2\alpha)$, and $U_1\cap C(-\tfrac{x_i}{\|x_i\|}, \frac{\pi}{2}-\arccos\tfrac1{\|x_i\|})\neq \emptyset$ for all $i\in \{1,\ldots k\}$.
 
    Let $U_2$ be a set of $N(n,\tfrac\pi2-\alpha)$ centers of caps of angular radius $\tfrac\pi2-\alpha$ covering the whole sphere. Since for all $i\in \{k+1,\ldots, m\}$ we have $C(-\tfrac{x_i}{\|x_i\|}, \frac{\pi}{2}-\alpha)\subseteq C(-\tfrac{x_i}{\|x_i\|}, \frac{\pi}{2}-\arccos\tfrac1{\|x_i\|})$, we also have $U_2\cap C(-\tfrac{x_i}{\|x_i\|}, \frac{\pi}{2}-\arccos\tfrac1{\|x_i\|})\neq \emptyset$ for $i\in \{k+1,\ldots, m\}$. Moreover, the positive hull of $U_2$ is $\E^{n}$. 
    
    By \cref{prop:illum}, $U_1\cup U_2$ illuminates $K$, so $I(K)\le M(n,2\alpha)+N(n,\tfrac\pi2-\alpha)$. 
\end{proof}

\begin{proof}[Proof of~\cref{cor:cap-bodies}.]
	The proof is based on Theorem~\ref{thm:cap-bodies-ill}, we only need to minimize $M(n,2\alpha)+N(n,\pi/2-\alpha)$, where $\theta\in(0,\pi/2)$.	
	
	We use an asymptotic upper bound on $M(n,\theta)$ due to Kabatjanskii and Levenstein~\cite{KL}. Namely, for fixed $0<\theta<\tfrac\pi2$ and $n\to\infty$, with
	\[
	KL(\theta):=\frac{1+\sin \theta}{2\sin\theta}\log_2 \frac{1+\sin \theta}{2\sin\theta}-\frac{1-\sin \theta}{2\sin\theta}\log_2 \frac{1-\sin \theta}{2\sin\theta},
	\]
	one has
	\begin{equation}\label{eqn:KLbound}
		\frac{1}{n}\log_2 M(n,\theta) \lesssim KL(\theta),
	\end{equation}
	where for positive $f(n),g(n)$, we write $f\lesssim g$ if $f(n)\le g(n)+o(1)$ as $n\to\infty$.

	With fixed $\theta$ and large $n$, by~(\ref{eq:BW}), 
	\begin{equation}\label{eqn:sphere-covering-bound}
		\frac{1}{n}\log_2 N(n,\pi/2 -\theta)\lesssim -\log_2\cos\theta.
	\end{equation}
	
    The functions $f(x)=KL(2x)$, $x\in(0,\pi/4)$, and $g(x)=-\log_2\cos x$, $x\in(0,\pi/4)$, are strictly decreasing/increasing, respectively, with $f(0+)-g(0+)=\infty$ and $f(\pi/4-)-g(\pi/4-)=-\frac{1}{2}$. Thus $f(x)=g(x)$, and equivalently \eqref{eqn:alpha}, has a unique solution $\alpha$. So by \cref{thm:cap-bodies-ill}, \eqref{eqn:KLbound} and~\eqref{eqn:sphere-covering-bound}, we obtain $I(K)\le (\tfrac1{\cos\alpha}+o(1))^n$. The numeric approximation for $\alpha$ can be obtained by use of computer. With $x=0.583808$, we computationally have $\max\{f(x),g(x)\}<1.19851$. 
\end{proof}

\begin{remark}
    The bound~\eqref{eqn:KLbound} can be improved for the range $\theta<\theta^*\approx 63^\circ$ using a connection between sizes of spherical codes for different angles, see~\cite{KL}*{Cor.~1}. Our optimal value of $\theta=2\alpha\approx 66.9^\circ$ is sufficiently far from this range so that using~\cite{KL}*{Cor.~1} instead of~\eqref{eqn:KLbound} does not improve the bound in \cref{cor:cap-bodies}.
\end{remark}

\section{Lower bound on $I_{c,n}^{(s)}$}\label{sec:IllumLB}

\begin{proof}[Proof of \cref{thm:lower-cap}.]
For any $\eps\in(0,\pi/6)$, by~\cite{ABP}*{Lemma~2} with $\phi=\pi/3+\eps$ and $\psi=\pi/3$, there exists a finite set $X\subset \S^{n-1}$ of cardinality $\ge \left(\frac{1+o(1)}{\sin(\pi/3+\eps)}\right)^n$ satisfying: (a) the angular distance between any two distinct points of $X$ is between $\pi/3$ and $2\pi/3$, and (b) every point of $\S^{n-1}$ is contained in at most $O(n\log n)$ spherical caps $C[x,\pi/3+\eps]$, $x\in X$. 
 
We now can take $Y=X\cap-X$ which still satisfies (b), and additionally satisfies (c): the angular distance between any two distinct points of $X$ is at least $\pi/3$. Now if $Y=\{x_i\}_{i=1}^m$, define $K:=\Sp[\tfrac{2}{\sqrt{3}}x_1,\dots,\tfrac{2}{\sqrt{3}}x_m]$. The cap associated with $\tfrac{2}{\sqrt{3}}x_i$ is then $C[x_i,\pi/6]$, so by (c) the spiky ball $K$ is a centrally symmetric cap body. A direction $u\in\S^{n-1}$ illuminates precisely those $\tfrac{2}{\sqrt{3}}x_i$ for which $-u\in C(x_i,\pi/3)$. Therefore, by (b), each direction illuminates at most $O(n\log n)$ vertices, and we obtain $I(K)\ge \left(\frac{1+o(1)}{\sin(\pi/3+\eps)}\right)^n/O(n\log n)$. Taking $\eps>0$ arbitrarily small completes the proof.
\end{proof}

\end{document}